\documentclass[11pt]{amsart}

\usepackage{amssymb,amsmath,epsfig,graphics}

\numberwithin{equation}{section}
\newcommand{\beq}{\begin{equation}}
\newcommand{\eeq}{\end{equation}}
\newcommand{\bea}{\begin{eqnarray}}
\newcommand{\eea}{\end{eqnarray}}
\newcommand{\beas}{\begin{eqnarray*}}
\newcommand{\eeas}{\end{eqnarray*}}

\newtheorem{theorem}{Theorem}[section]

\newtheorem{proposition}[theorem]{Proposition}
\newtheorem{corollary}[theorem]{Corollary}

\newtheorem{remark}[theorem]{Remark}
\newtheorem{example}[theorem]{Example}
\newtheorem{examples}[theorem]{Examples}
\newtheorem{foo}[theorem]{Remarks}

\newcommand{\bM}{\mathbb M}

\title[Eigenvalues in positive Ricci curvature] {A note on lower bounds estimates for the Neumann eigenvalues of  manifolds with positive Ricci curvature}

\author{Fabrice Baudoin}
\address{Department of Mathematics\\Purdue University \\
West Lafayette, IN 47907} \email[Fabrice Baudoin]{fbaudoin@math.purdue.edu}
\thanks{First author supported in part by
NSF Grant DMS 0907326}

\author{Alice Vatamanelu}
\address{Department of Mathematics\\Purdue University \\
West Lafayette, IN 47907} \email[Alice Vatamanelu]{vatamanelu@gmail.com}

\begin{document}

\maketitle

\begin{abstract}
We study new heat kernel estimates for the Neumann heat kernel on a  compact manifold with positive Ricci curvature and convex boundary. As a consequence, we obtain  lower bounds for the Neumann eigenvalues which are consistent with Weyl's asymptotics.
\end{abstract}

\tableofcontents

\section{Introduction}

Eigenvalues of compact Riemannian  manifolds have been extensively studied (see for instance  Chavel \cite{Chavel}, Cheng \cite{Cheng},   Li-Yau \cite{LY2}, \cite{LY}, and the references therein). In particular, it has been proved by Li and Yau \cite{LY} that for the Neumann eigenvalues of a compact Riemannian manifold with non negative Ricci curvature and convex boundary 
\[
\lambda_k \ge C(n) \frac{k^{2/n}}{D(\bM)^2},
\]
where $C(n)$ is a constant that only depends on the dimension $n$ of the manifold and where $D(\bM)$ is the diameter of $\bM$. These lower bound estimates are obtained by proving an on-diagonal upper bound for the Neumann heat kernel.  In this note, we follow the approach of Li and Yau, but use the tools introduced in Bakry-Qian \cite{BQ}, Bakry-Ledoux  \cite{BL}  and Baudoin-Garofalo \cite{BG1}  to prove new upper bounds for the Neumann heat kernel in the case where the Ricci curvature is bounded from below by a positive constant  $\rho$. These new heat kernel upper bounds lead to lower bounds of the form
\[
\lambda_k \ge C_1(n,\rho,k),
\]
and 
\[
\lambda_k \ge C_2 (n,\rho,D(\bM), k)
\]
where $C_1(n,\rho,k)$ and $C_2 (n,\rho,D(\bM), k)$   have order $k^{2/n}$ when $k \to \infty$, which is  consistent with Weyl's aymptotics.

\section{Li-Yau type estimates on manifolds with positive Ricci curvature and convex boundary}

\subsection{The Neumann semigroup}

Let $\mathbb{M}$ be a $n$-dimensional smooth, compact and connected Riemannian manifold with  boundary $\partial \mathbb{M}$. Let us denote by $N$ the outward unit vector field on $\partial M$. The second fundamental form of $\partial \mathbb{M}$ is defined on vector fields tangent to $\partial \mathbb{M}$ by
\[
\Pi(X,Y)=\langle \nabla_N X, Y \rangle.
\]
The boundary is then said to be convex if $\Pi \ge 0$ as a symmetric bilinear form. Throughout this paper, we  will assume that the boundary $\partial \mathbb{M}$ is convex.  We shall moreover assume that the Ricci curvature tensor of $\bM$ satisfies $\mathbf{Ric}  \ge \rho $ for some $\rho >0$.

\

Let $\Delta$ be the Laplace-Beltrami operator of $\bM$, with the sign convention that makes $\Delta$ a non positive symmetric operator on $C^\infty_0(\bM)$. It is well-known that $\Delta$ is essentially self-adjoint  on 
\[
\mathcal{D}=\left\{ f \in C^\infty(\bM), Nf=0 \text{ on } \partial \mathbb{M}\right\}.
\] 
The Friedrichs extension of $\Delta$ is then the generator of strongly continuous Markov semigroup which is called the Neumann semigroup. We shall denote this semigroup by $(P_t)_{t \ge 0}$.

By ellipticity of $\Delta$, for every $f \in L^p (\bM)$, $1\le p \le +\infty$, $P_t f \in \mathcal{D}$, $t>0$ and
\[
\frac{\partial P_t f}{\partial t}= \Delta P_t f.
\]
Also (see for instance  \cite{SV}), if $f \in C^{2} (\bM)$ is such that $Nf \le 0$ on $ \partial \bM$, then
\begin{align}\label{Submartingale}
\frac{\partial P_t f}{\partial t} \ge   P_t \left(\Delta f \right).
\end{align}

Moreover  $(P_t)_{t \ge 0}$ has  a smooth heat kernel, that is there exists a smooth function  $p:(0,+\infty)\times \bM \times \bM \to (0,+\infty)$ such that for every $f \in L^{\infty}(\bM)$: 
\[
P_t f(x)=\int_{\bM} p(t,x,y)f(y)d\mu(y).
\]

\subsection{Li-Yau gradient type estimate and heat kernel bounds}

\begin{theorem}\label{LIYAU}
 Let $f \in C^2(\bM)$, $f>0$.  For $t  > 0$, and $x \in \bM$,
\[
\| \nabla \ln P_t f (x) \|^2 \le e^{-\frac{2\rho t}{3}}  \frac{  \Delta P_t f (x)}{P_t f (x)} +\frac{n\rho}{3} \frac{e^{-\frac{4\rho t}{3}}}{ 1-e^{-\frac{2\rho t}{3}}}.
 \]
\end{theorem}

\begin{proof}
Consider the functional
\[
\Phi (t,x)= a(t) (P_t f)(x)\| \nabla \ln P_t f  (x)\|^2 , \quad t > 0, x \in  \bM,
\]
where
\[
a(t)=e^{\frac{2\rho t}{3}}\left(e^{\frac{2\rho t}{3}} - 1 \right)^2.
\]
Since $f \in C^2(\bM)$, let  us first observe that according to Qian \cite{Qi}, $\| \nabla P_t f \|^2 (x) \le e^{-2\rho t} P_t \| \nabla f \|^2 (x)$, so that we have, uniformly on $\bM$,
\begin{align}\label{limit condition}
\lim_{t \to 0} \Phi (t,x)=0.
\end{align}
We now compute
\begin{align}\label{variation}
\frac{\partial  \Phi}{\partial t}= L\Phi+\frac{a'(t)}{a(t)}  \Phi - a(t) (P_t f)(x) \left( \Delta \| \nabla \ln  P_t f  \|^2-2\langle \nabla \ln  P_t f, \nabla \Delta \ln P_t f \rangle \right).
\end{align}
This computation made be performed by using the so-called $\Gamma_2$-calculus developed in \cite{BE}. More precisely, denote for functions $u$ and $v$,
\[
\Gamma(u,v)=\frac{1}{2}\left( \Delta(uv)-u\Delta v-v \Delta u \right)=\langle \nabla u,\nabla v \rangle,
\]
and
\[
\Gamma_2(u,v)=\frac{1}{2}\left( \Delta \Gamma(u,v)-\Gamma(u, \Delta v)-\Gamma(v, \Delta u) \right).
\]
Using then the change of variable formula (see \cite{B} or \cite{BL} ),
\[
\Gamma_2(\ln u,\ln u)=\frac{1}{u^2} \Gamma_2 (u,u)-\frac{1}{u^3} \Gamma(u, \Gamma(u,u))+\frac{1}{u^4}\Gamma(u,u)^2
\]
and $ \frac{\partial P_t f}{\partial t}= \Delta P_t f$ yields (\ref{variation}). Now, from Bochner's formula, we have
\[
 \Delta \| \nabla \ln  P_t f  \|^2-2\langle \nabla \ln  P_t f, \nabla \Delta \ln P_t f \rangle =2 \| \nabla^2  \ln  P_t f  \|^2+2\mathbf{Ric}(\nabla  \ln  P_t f ,\nabla  \ln  P_t f ),
\]
and Cauchy-Schwarz inequality implies,
\[
\| \nabla^2  \ln  P_t f  \|^2 \ge \frac{1}{n} (\Delta  \ln  P_t f )^2.
\]
Since by assumption
\[
\mathbf{Ric}(\nabla \ln P_t f,\nabla \ln P_t f) \ge \rho \| \nabla \ln P_t f \|^2,
\]
 we obtain therefore
\begin{align}\label{hj}
\frac{\partial  \Phi}{\partial t} \le  \Delta \Phi+ L\Phi+\frac{a'(t)}{a(t)}  \Phi - a(t) (P_t f)(x) \left( \frac{2}{n} (\Delta \ln P_t f)^2 + 2\rho \| \nabla \ln P_t f \|^2 \right).
\end{align}
Now, observe that for every $\gamma \in \mathbb{R}$,
\[
(\Delta \ln P_t f)^2 \ge 2 \gamma \Delta \ln P_t f-\gamma^2=2\gamma \frac{\Delta P_t f}{P_t f} -2\gamma \| \nabla \ln P_t f \|^2 -\gamma^2.
\]
In particular, by chosing 
\[
\gamma(t)=-\frac{n \rho}{3}\frac{1}{e^{\frac{2\rho t}{3}} - 1} ,
\]
so that
\[
\frac{a'(t)}{a(t)}+\frac{4}{n}\gamma(t)-2\rho=0
\]

and then  coming back to (\ref{hj}), we infer
\[
\frac{\partial  \Phi}{\partial t} \le  \Delta \Phi -\frac{4 a(t)\gamma (t) }{n}  \Delta P_t f +2\frac{a(t)\gamma(t)^2 }{n} P_t f.
\]
We now make the crucial observation that on $\partial \bM$, 
\[
N\left( (P_t f)\| \nabla \ln P_t f  \|^2\right)=N \left( \frac{\| \nabla P_t f  \|^2}{P_t f}\right)= -\frac{NP_t f}{(P_t f)^2} \| \nabla P_t f  \|^2 +\frac{N \| \nabla P_t f  \|^2}{P_t f}=\frac{N \| \nabla P_t f  \|^2}{P_t f},
\]
and, that by the  convexity assumption,
\[
N \| \nabla P_t f  \|^2 =-2\Pi (\nabla P_t f , \nabla P_t f) \le 0.
\]
As a conclusion, on $\partial \bM$, we have
\begin{align}\label{direction}
N\Phi \le 0.
\end{align}
We fix now $T>0$, $x \in \bM$ and consider
\[
\Psi(t)= (P_{T-t} \Phi) (x).
\]
As a consequence of ($\ref{Submartingale}$) and (\ref{direction}), we thus get
\begin{align*}
\Psi'(t)&  \le P_{T-t} \left( \frac{\partial \Phi}{\partial t} -\Delta \Phi \right)(x) \\
 &  \le P_{T-t} \left(-\frac{4 a(t)\gamma (t) }{n}  \Delta P_t f +2\frac{a(t)\gamma(t)^2 }{n} P_t f \right)(x) \\
  & \le -\frac{4 a(t)\gamma (t) }{n}   \Delta P_T f (x) +2\frac{a(t)\gamma(t)^2 }{n}  P_T f(x)
\end{align*}
We now integrate the previous inequality from $0$ to $T$, use (\ref{limit condition}), and end up with
\[
\Phi(T,x) \le -\int_0^T \frac{4\gamma (t) }{n}a
(t) dt \Delta P_T f (x) +\frac{2}{n} \int_0^T \gamma(t)^2 a(t) dt  P_T f(x).
\]
Since 
\[
a(t)=e^{\frac{2\rho t}{3}}\left(e^{\frac{2\rho t}{3}} - 1 \right)^2.
\]
the conclusion is reached by computing 
\[
\int_0^T  a(t)\gamma (t) dt=-\frac{n }{4} \left( e^{\frac{2\rho T}{3}} -1\right)^2
\]
and 
\[
\int_0^Ta(t)\gamma (t)^2 dt=\frac{n^2 \rho}{6} \left(e^{\frac{2\rho T}{3}} - 1 \right)
\]
\end{proof}

\begin{remark}

\

\begin{itemize}
\item In \cite{BG1},  in the case where the manifold has no boundary,  the same inequality was obtained  as a by product of a class of more general Li-Yau type inequalities.
\item In the case $\rho =0$, considering the functional 
\[
\Phi (t,x)=t^2 (P_t f)(x)\| \nabla \ln P_t f  (x)\|^2 , \quad t \ge 0, x \in  \bM,
\]
would lead to the celebrated Li-Yau inequality for the Neumann semigroup on manifolds with convex boundaries (see \cite{LY}, \cite{BQ}):
\[
\| \nabla \ln P_t f (x) \|^2 \le   \frac{  \Delta P_t f (x)}{P_t f (x)} +\frac{n}{2t} .
\]
\item In the case $\rho=0$, a Li-Yau type inequality is obtained in \cite{W} without the assumption that the boundary is convex.  
\end{itemize}
\end{remark}

\subsection{Harnack inequality}

As is well-known since Li-Yau \cite{LY}, gradients bounds like Theorem \ref{LIYAU}, imply by integrating along geodesics a Harnack inequality for the heat semigroup:

\begin{theorem}
Let $f \in L^\infty(\bM)$, $f >0$. For $0\le s<t$ and $x,y \in \bM$,
\[
 P_s f(x) \le  \left( \frac{1-e^{-\frac{2\rho t}{3}}}{1-e^{-\frac{2\rho s}{3}} }\right)^{n/2} e^{\frac{\rho}{6} \frac{d(x,y)^2}{ e^{\frac{2\rho t}{3}}-e^{\frac{2\rho s}{3} }}  }P_t f (y).
\]
\end{theorem}

\begin{proof}
We first assume that $f \in C^2(\bM)$. Let $x,y \in \bM$ and let $\gamma:[s,t] \to \bM$, $s<t$ be an absolutely continuous path such that $\gamma(s)=x, \gamma(t)=y$.
We write Theorem \ref{LIYAU} in the form
\begin{align}\label{LIYAU2}
\| \nabla \ln P_u f (x) \|^2 \le a(u) \frac{  \Delta P_u f (x)}{P_u f (x)} +b(u),
 \end{align}
 where 
 \[
 a(u)=e^{-\frac{2\rho u}{3}}, 
 \]
 and 
 \[
 b(u)=\frac{n\rho}{3} \frac{e^{-\frac{4\rho u}{3}}}{ 1-e^{-\frac{2\rho u}{3}}}.
 \]
 Let us now consider
 \[
 \phi(u)=\ln P_u f(\gamma(u)).
 \]
 We  compute
 \[
  \phi(u)= ( \partial_u \ln P_u  f) (\gamma(u))+\langle \nabla \ln P_u f (\gamma(u)),\gamma'(u) \rangle.
 \]
 Now, for every $\lambda >0$, we have
 \[
 \langle \nabla \ln P_u f (\gamma(u)),\gamma'(u) \rangle \ge -\frac{1}{2\lambda^2} \| \nabla \ln P_u f (x) \|^2 -\frac{\lambda^2}{2} \| \gamma'(u) \|^2.
 \]
 Choosing  $\lambda=\sqrt{\frac{a(u)}{2} }$ and using then (\ref{LIYAU2}) yields
 \[
 \phi'(u) \ge -\frac{a(u)}{b(u)} -\frac{1}{4} a(u) \| \gamma'(u) \|^2.
 \]
 By integrating this inequality from $s$ to $t$ we get as a result.
 \[
 \ln P_tf(y)-\ln P_s f(x)\ge -\int_s^t \frac{a(u)}{b(u)} du  -\frac{1}{4} \int_s^t a(u) \| \gamma'(u) \|^2 du.
 \]
We now minimize the quantity  $\int_s^t a(u) \| \gamma'(u) \|^2 du$ over the set of absolutely continuous paths such that $\gamma(s)=x, \gamma(t)=y$. By using reparametrization of paths, it is seen that
 \[
 \int_s^t a(u) \| \gamma'(u) \|^2 du \le \frac{d^2(x,y)}{\int_s^t \frac{dv}{a(v)}},
 \]
 with equality achieved for $\gamma(u)=\sigma\left( \frac{\int_s^u \frac{dv}{a(v)}}{\int_s^t \frac{dv}{a(v)}} \right)$ where $\sigma:[0,1] \to \bM$ is a unit geodesic joining $x$ and $y$. As a conclusion,
 \[
 P_sf(x) \le \exp\left( \int_s^t \frac{a(u)}{b(u)} du + \frac{d^2(x,y)}{4\int_s^t \frac{dv}{a(v)}}  \right) P_tf(y).
 \]
 Using finally the expressions of $a$ and $b$ leads to  
 \[
 P_s f(x) \le  \left( \frac{1-e^{-\frac{2\rho t}{3}}}{1-e^{-\frac{2\rho s}{3}} }\right)^{n/2} e^{\frac{\rho}{6} \frac{d(x,y)^2}{ e^{\frac{2\rho t}{3}}-e^{\frac{2\rho s}{3} }}  }P_t f (y).
 \]
 If $f \in L^\infty(\bM)$ but  $ f \notin C^2(\bM)$, in the previous argument we replace $f$ by $P_\tau f$, $\tau >0$ and, at the end,  let $\tau \to 0$.
\end{proof}

As a straightforward corollary, we get a Harnack inequality for the Neumann heat kernel:

\begin{corollary}\label{HK}
Let $p(t,x,y)$ be the Neumann heat kernel of $\bM$.  For $0< s<t$ and $x,y,z \in \bM$,
\[
p(s,x,y) \le  \left( \frac{1-e^{-\frac{2\rho t}{3}}}{1-e^{-\frac{2\rho s}{3}} }\right)^{n/2} e^{\frac{\rho}{6} \frac{d(y,z)^2}{ e^{\frac{2\rho t}{3}}-e^{\frac{2\rho s}{3} }}  }p(t,x,z).
\]
\end{corollary}

\subsection{On diagonal heat kernel estimates}

We now prove on-diagonal heat kernel estimates for the Neumann heat kernel that stem from the previous  Harnack inequalities. We shall essentially focus on two types of estimates: Estimates that only depend on the curvature parameter $\rho$ or estimates that depend on $\rho$ and the diameter of $\bM$.

\begin{proposition}\label{jhn}
Let $p(t,x,y)$ be the Neumann heat kernel of $\bM$. For $t>0$, $x \in \bM$,
\[
\left( \frac{\rho}{6\pi} \right)^{n/2} \frac{1}{\left( 1-e^{-\frac{2\rho t}{3}} \right)^{n/2} } \le p(t,x,x) \le \frac{1}{\mu(\bM)}  \frac{1}{\left( 1-e^{-\frac{2\rho t}{3}} \right)^{n/2} }.
\]
\end{proposition}

\begin{proof}
From Corollary \ref{HK},  for $0\le s<t$ and $x,y \in \bM$,
\begin{align}\label{HK2}
 p(s,x,x) \le  \left( \frac{1-e^{-\frac{2\rho t}{3}}}{1-e^{-\frac{2\rho s}{3}} }\right)^{n/2} p(t,x,x).
\end{align}
We have $\lim_{t \to +\infty} p(t,x,x)= \frac{1}{\mu(\bM)}$. Thus by letting $t \to +\infty$ in (\ref{HK2}), we get
\[
 p(s,x,x) \le \frac{1}{\mu(\bM)}  \frac{1}{\left( 1-e^{-\frac{2\rho s}{3}} \right)^{n/2} }.
\]
On the other hand, $\lim_{s \to 0} p(s,x,x) (4\pi s)^{n/2} = 1$, so by letting $s \to 0$ in (\ref{HK2}), we  deduce 
\[
p(t,x,x) \ge \left( \frac{\rho}{6\pi} \right)^{n/2} \frac{1}{\left( 1-e^{-\frac{2\rho t}{3}} \right)^{n/2} } . 
\]
\end{proof}

\begin{remark}
Interestingly, Proposition \ref{jhn} contains the geometric bound 
\[
\mu(\bM) \le \left( \frac{6\pi}{\rho} \right)^{n/2}.
\] 
This bound is not sharp since from the Bishop's volume comparison theorem the volume of $\bM$ is less than the volume of the $n$-dimensional sphere with radius $\sqrt{\frac{n-1}{\rho} }$ which is $ \frac{2 \pi^{\frac{n+1}{2} }}{\Gamma \left( \frac{n+1}{2}\right)} \left(\frac{n-1}{\rho} \right)^{n/2}$. However, by using Stirling's equivalent we observe that, when $n \to \infty$,  the ratio between $ \left( \frac{6\pi}{\rho} \right)^{n/2}$  and $ \frac{2 \pi^{\frac{n+1}{2} }}{\Gamma \left( \frac{n+1}{2}\right)} \left(\frac{n-1}{\rho} \right)^{n/2}$ only  has order $n \left( \frac{3}{e} \right)^{n/2}$.
\end{remark}

Since
\[
\int_{\bM} p(t,x,x) d\mu(x)=\sum_{k=0}^{+\infty} e^{-\lambda_k t},
\]
where $0=\lambda_0<\lambda_1 \le \lambda_2 \le \cdots \lambda_n \le  \cdots$ are the Neumann eigenvalues of $\bM$, we deduce  from the previous estimates
\begin{align}\label{estimate}
\left( \frac{\rho}{6\pi} \right)^{n/2} \frac{\mu(\bM)}{\left( 1-e^{-\frac{2\rho t}{3}} \right)^{n/2} } \le \sum_{k=0}^{+\infty} e^{-\lambda_k t} \le  \frac{1}{\left( 1-e^{-\frac{2\rho t}{3}} \right)^{n/2} }.
\end{align}

\begin{proposition}
Let $p(t,x,y)$ be the Neumann heat kernel of $\bM$. For $t>0$, $x \in \bM$,
\[
p(t,x,x)   \le \frac{\left(1+e^{-\frac{2\rho t}{3} } \right)^{n/2} e^{n/2}  }{ \mu \left(  B \left(x, \sqrt{ r(t)} \right) \right) },
\]
with $r(t)=\frac{3n}{\rho} \left( e^{\frac{4\rho t}{3}}-e^{\frac{2\rho t}{3} } \right)$.
\end{proposition}

\begin{proof}
From Corollary \ref{HK},
\[
p(t,x,x) \le  \left( \frac{1-e^{-\frac{4\rho t}{3}}}{1-e^{-\frac{2\rho t}{3}} }\right)^{n/2} e^{\frac{\rho}{6} \frac{d(x,y)^2}{ e^{\frac{4\rho t}{3}}-e^{\frac{2\rho t}{3} }}  }p(2t,x,y).
\]
Thus, for $y \in B \left(x, \sqrt{ \frac{3n}{\rho} \left( e^{\frac{4\rho t}{3}}-e^{\frac{2\rho t}{3} } \right)} \right)$,
\[
p(t,x,x) \le \left(1+e^{-\frac{2\rho t}{3} } \right)^{n/2} e^{n/2} p(2t,x,y).
\]
Integrating with respect to $y$ over the ball  $ B \left(x, \sqrt{ \frac{3n}{\rho} \left( e^{\frac{4\rho t}{3}}-e^{\frac{2\rho t}{3} } \right)} \right)$ therefore yields
\begin{align*}
p(t,x,x) & \le \frac{\left(1+e^{-\frac{2\rho t}{3} } \right)^{n/2} e^{n/2}  }{ \mu \left(  B \left(x, \sqrt{ \frac{3n}{\rho} \left( e^{\frac{4\rho t}{3}}-e^{\frac{2\rho t}{3} } \right)} \right)\right) } \int_{ B \left(x, \sqrt{ \frac{3n}{\rho} \left( e^{\frac{4\rho t}{3}}-e^{\frac{2\rho t}{3} } \right)} \right) } p(2t,x,y)\mu(dy) \\
 &  \le \frac{\left(1+e^{-\frac{2\rho t}{3} } \right)^{n/2} e^{n/2}  }{ \mu \left(  B \left(x, \sqrt{ \frac{3n}{\rho} \left( e^{\frac{4\rho t}{3}}-e^{\frac{2\rho t}{3} } \right)} \right)\right) }
\end{align*}
\end{proof}

Combining the previous estimate with the Bishop-Gromov comparison theorem yields the following upper bound estimate for the heat kernel, which in small times, may be better than the upper bound of Proposition \ref{jhn}. 

\begin{corollary}\label{lk}
Let $p(t,x,y)$ be the Neumann heat kernel of $\bM$. Denote $D(\bM)$ the diameter of $\bM$ and consider
\[
\tau=\frac{3}{2\rho} \ln \left( \frac{ 1+\sqrt{1+\frac{ 4\rho D(\bM)^2}{3n } } }{2} \right)
\]

For $x \in \bM$:
\begin{itemize}
\item If $0 < t \le \tau$,
\[
p(t,x,x)   \le\frac{\left(1+e^{-\frac{2\rho t}{3} } \right)^{n/2} e^{n/2}  }{ \mu \left(  \bM \right) } \frac{V_\rho ( D(\bM) )}{V_\rho ( \sqrt{r(t)} ) },
\]
with $r(t)=\frac{3n}{\rho} \left( e^{\frac{4\rho t}{3}}-e^{\frac{2\rho t}{3} } \right)$.
\item If $t \ge \tau$,
\[
p(t,x,x)   \le \frac{\left(1+e^{-\frac{2\rho t}{3} } \right)^{n/2} e^{n/2}  }{ \mu \left(  \bM \right) },
\]
\end{itemize}
where
\[
V_\rho ( s)= \int_0^s \sin^{n-1} \left( \sqrt{\frac{\rho}{n-1} }u \right) du.
\]
\end{corollary}

\section{Lower bounds for the eigenvalues}

Heat kernel upper bounds are a well-known device to prove lower bounds on the spectrum (see \cite{DL}, \cite{LY}). We therefore apply the results of the previous Section and obtain:

\begin{theorem}
Let $0=\lambda_0<\lambda_1 \le \lambda_2 \le \cdots \lambda_k \le  \cdots$ be the Neumann eigenvalues of $\bM$. 
 \begin{itemize}
\item For every $k \in \mathbb{N}$,
\[
\lambda_k \ge -\frac{ n \rho}{3 \ln \left( 1-\frac{1}{ (1+e^{-n/2}k)^{2/n} } \right) }
\]
\item For every $k \in \mathbb{N}$, $k >2^{n/2} e^n-e^{n/2}$,
\[
\lambda_k \ge \frac{ n \rho}{3 \ln \left( \frac{ 1+\sqrt{1+\frac{4\rho}{3n} V_\rho^{-1} \left( \frac{(2e)^{n/2}}{1+ke^{-n/2} } V_\rho(D(\bM))\right)^2 } }{2} \right) }
\]
\end{itemize}

\end{theorem}

\begin{proof}

\

\begin{itemize}

\item Thanks to (\ref{estimate}), we have for every $t >0$,
\[
1+ke^{-\lambda_k t} \le \sum_{k=0}^{+\infty} e^{-\lambda_k t} \le \frac{1}{\left( 1-e^{-\frac{2\rho t}{3}} \right)^{n/2} }.
\]
Choosing then $t= \frac{n}{2\lambda_k}$ yields the lower bound
\[
\lambda_k \ge -\frac{ n \rho}{3 \ln \left( 1-\frac{1}{ (1+e^{-n/2}k)^{2/n} } \right) }.
\]

\item Thanks to Corollary \ref{lk}, if $t \le \tau=\frac{3}{2\rho} \ln \left( \frac{ 1+\sqrt{1+\frac{ 4\rho D(\bM)^2}{3n } } }{2} \right)$, then
\[
1+ke^{-\lambda_k t} \le \frac{ (2e)^{n/2}  V_\rho ( D(\bM) )}{V_\rho ( \sqrt{r(t)} ) }.
\]
And, if $t \ge \tau$, then
\[
1+ke^{-\lambda_k t} \le (2e)^{n/2} .
\]
For $k >2^{n/2} e^n-e^{n/2}$, we have
\[
1+ke^{-n/2} >  (2e)^{n/2} 
\]
and thus $\frac{n}{2 \lambda_k}  \le \tau$. This implies
\[
1+ke^{- n/2} \le \frac{ (2e)^{n/2}  V_\rho ( D(\bM) )}{V_\rho \left( \sqrt{r \left( \frac{n}{2\lambda_k} \right)} \right) },
\]
and the result follows by direct computations.
\end{itemize}
\end{proof}

\begin{remark}
When $k \to \infty$, we have
\[
-\frac{ n \rho}{3 \ln \left( 1-\frac{1}{ (1+e^{-n/2}k)^{2/n} } \right) }\sim_{k \to \infty} \frac{ n \rho }{3e } k^{2/n}
\]
and
\[
\frac{ n \rho}{3 \ln \left( \frac{ 1+\sqrt{1+\frac{4\rho}{3n} V_\rho^{-1} \left( \frac{(2e)^{n/2}}{1+ke^{-n/2} } V_\rho(D(\bM))\right)^2 } }{2} \right) }
\sim_{k \to +\infty} \frac{n}{2e^2} \left( \frac{n\rho}{n-1} \right)^{1-\frac{1}{n}} \left( \frac{k}{V_\rho(D(\bM) ) }\right)^{2/n},
\]
which is consistent with the Weyl asymptotics
\[
\lambda_k^{n/2} \sim_{k \to \infty} \frac{(4 \pi)^{n/2} \Gamma \left( 1+\frac{n}{2} \right)}{ \mu(\bM)} k.
\]
\end{remark}

\end{document}